\documentclass[12pt]{amsart}
\usepackage[margin=.8in]{geometry}
\usepackage{amssymb}
\usepackage{amsmath}
\usepackage{latexsym}
\usepackage{mathrsfs}
\usepackage{amsthm}
\usepackage{graphicx}
\usepackage{caption}
\usepackage{dsfont}
\usepackage{hyperref}
\usepackage{enumitem}
\usepackage{ulem}
\usepackage{cancel}
\usepackage{color}
\usepackage{enumerate}
\usepackage{comment}

\numberwithin{equation}{section}
\newtheorem{theorem}[equation]{Theorem}

\newtheorem{lemma}[equation]{Lemma}

\newtheorem*{acknowledge}{Acknowledgments}

\theoremstyle{definition}
\newtheorem{rem}[equation]{Remark}
\newtheorem{example}[equation]{Example}

\newcommand{\la}{\langle}
\newcommand{\ra}{\rangle}
\newcommand{\diver}{{\rm div}}
\newcommand{\vol}{{\rm vol}}

\newcommand{\Ric}{{\rm Ric}}
\newcommand{\Span}{{\rm Span}}

\newcommand{\R}{\mathbb{R}}
\newcommand{\cH}{\mathcal{H}}
\newcommand{\N}{\mathbb{N}}

\newcommand{\be}{\begin{equation*}}
\newcommand{\ee}{\end{equation*}}
\newcommand{\bff}{\begin{proof}}
\newcommand{\ef}{\end{proof}}
\newcommand{\ben}[1]{\begin{equation}\label{#1}}
\newcommand{\een}{\end{equation}}
\newcommand{\bq}{\begin{eqnarray*}}
\newcommand{\bqn}[1]{\begin{eqnarray}\label{#1}}
\newcommand{\eq}{\end{eqnarray*}}
\newcommand{\eqn}{\end{eqnarray}}

\def\sideremark#1{\ifvmode\leavevmode\fi\vadjust{\vbox to0pt{\vss
 \hbox to 0pt{\hskip\hsize\hskip1em                        
  \vbox{\hsize2cm\tiny\raggedright\pretolerance10000         
  \noindent #1\hfill}\hss}\vbox to8pt{\vfil}\vss}}}

\begin{document}
\title[]{A note on Kuttler-Sigillito's inequalities}

\author{Asma Hassannezhad}
\address
{Asma Hassannezhad: University of Bristol,
School of Mathematics,
University Walk,
Bristol BS8 1TW, UK}
\email{asma.hassannezhad@bristol.ac.uk}
\author{Anna Siffert}
\address
{Anna Siffert: Max Planck Institute for Mathematics,
Vivatsgasse 7, 53111 Bonn, Germany}
\email{siffert@mpim-bonn.mpg.de}
\date{\today}

\begin{abstract} 
We provide several inequalities between eigenvalues of some classical eigenvalue problems on domains
with $C^2$ boundary 
 in complete Riemannian manifolds.
A key tool in the proof is the generalized Rellich identity on a Riemannian manifold. Our results in particular extend some inequalities due to Kutller and Sigillito from subsets of $\R^2$ to the manifold setting.
\end{abstract}
\maketitle

\section{Introduction}
The objective of this manuscript is to establish several inequalities between eigenvalues of 
the classical eigenvalue problems mentioned below. Let $(M,g)$ be a complete Riemannian manifold of dimension $n\ge2$ and   $\Omega$  be a bounded domain in $M$ with nonempty $C^2$ boundary $\partial \Omega$. The eigenvalue problems we consider include the  Neumann and Dirichlet eigenvalue problems on $\Omega$:
 \begin{align}
    &\label{dir}  \left\{\begin{array}{ll}\Delta u+\lambda u=0
      \quad \;\; &\mbox{in}\;\; \Omega,\\
  u=0\quad \;\;  &  \mbox{on}\;\; \partial
  \Omega, \end{array} \right. &\qquad\qquad \mbox{Dirichlet eigenvalue problem\,,}\\
  &\left\{ \begin{array}{ll} \Delta u+\mu u=0
      \quad \;\; &\mbox{in}\;\; \Omega,\\
  {\partial_\nu u}=0\quad \;\;  &
   \mbox{on}\;\; \partial \Omega, \end{array} \right.& \qquad\qquad \mbox{Neumann eigenvalue problem\,,}\label{neu} 
  \end{align}
     \noindent where $\Delta={\rm div}\nabla$ is the Laplace--Beltrami operator, $\nu$ is the unit outward normal vector on $\partial \Omega$, and $\partial_\nu$ denotes the outward normal derivative.  
     The Dirichlet eigenvalues describe the fundamental modes of vibration of an idealized drum, and
the Neumann eigenvalues appear naturally in the study of the vibrations of a free membrane; see e.g. \cite{Ber86,Cha84}.
 
 \smallskip 
  
We also consider the Steklov eigenvalue problem, which is an eigenvalue problem with the spectral parameter in the boundary conditions:
  \begin{align}  & \left\{ \begin{array}{ll} \Delta u=0
      \quad \;\; &\mbox{in}\;\; \Omega,\\
  {\partial_\nu u}=\sigma u\quad \;\;  &
   \mbox{on}\;\; \partial \Omega, \end{array} \right.& \qquad\qquad \mbox{Steklov eigenvalue problem\,.}\label{stek} 
     \end{align}
     The Steklov eigenvalues encode the squares of the natural
frequencies of vibration of a thin membrane with free frame, whose
mass is uniformly distributed at the boundary; see the recent survey paper~\cite{GP17} and references therein.

\smallskip

The last set of eigenvalue problems we consider are the so-called  Biharmonic Steklov problems:     
     \begin{align}
  & \label{bihar-i} \left\{ \begin{array}{ll} \Delta^2 u=0 \quad \;\; &\mbox{in}\;\; \Omega,\\
  u=\Delta u-\eta{\partial_\nu u}=0 \quad \;\;  &  \mbox{on}\;\;
   \partial \Omega, \end{array} \right. &\qquad\qquad \mbox{Biharmonic Steklov problem I\,;}&\\
  & \label{bihar-ii} \left\{ \begin{array}{ll} \Delta^2 u=0 \quad \;\; &\mbox{in}\;\; \Omega,\\
  {\partial_\nu u}={\partial_\nu \Delta u}+\xi u=0 \quad \;\;  &  \mbox{on}\;\;
   \partial \Omega,\end{array} \right. &\qquad\qquad \mbox{Biharmonic Steklov problem II\,.}
   \end{align}
   \smallskip
   The eigenvalues problems \eqref{bihar-i} and \eqref{bihar-ii} play an important role in biharmonic analysis and elastic mechanics. We refer the reader to \cite{FGW05,BFG09,Liu(a)11,Liu16} for some recent results on eigenvalue estimates of problem~\eqref{bihar-i}. Moreover, a physical interpretation of problem \eqref{bihar-i} can be found in \cite{FGW05,Liu(a)11}. Problem \eqref{bihar-ii} was first studied in \cite{KS68,Kut82} where the main focus was on the first nonzero eigenvalue, which appears as an optimal constant in a priori inequality; see \cite{Kut82} for more details. 
   \smallskip 
   
It is well-known that the spectra of the eigenvalue problems \eqref{neu}--\eqref{bihar-ii} are discrete and nonnegative. We may thus arrange their eigenvalues in  increasing order, where we repeat an eigenvalue as often as its multiplicity requires. 
The $k$-th eigenvalue of one of the above eigenvalue problems will be denoted by the corresponding letter for the eigenvalue with a subscript $k$, e.g.~the $k$-th Neumann eigenvalue will be denoted by $\mu_k$.
Note that $\mu_1=\sigma_1=\xi_1=0$.

\smallskip

There is a variety of literature on the study of bounds on the eigenvalues of each problem mentioned above in terms of the geometry of the underlying space \cite{LY80,Liu(a)11,RS16,GP17}. However, instead of studying each eigenvalue problem individually, it is also interesting to explore relationships and inequalities between eigenvalues of different eigenvalue problems. Among this type of results, one can mention  the relationships between the Laplace  and Steklov eigenvalues  studied in \cite{WX09,LP15,PS16}, and
various inequalities between the first nonzero eigenvalue of problems \eqref{neu}--\eqref{bihar-ii}  on bounded domains of $\R^2$ obtained  by Kuttler and Sigilito in \cite{KS68}; see Table\,1 (Note that there was a misprint in Inequality VI  in \cite{KS68}. The correct version of the inequality is stated in Table\,1.). 
\begin{center}
\begin{table}[h]
\caption{Inequalities obtained by Kuttler and Sigillito in \cite{KS68}.}
\begin{tabular}{ |c|c|c|c|}
 \hline
 Inequalities & Conditions on $\Omega\subset\R^2$&Special case of  \\
 \hline
   $\mu_2\sigma_2\leq\xi_2$    &&Thm. \ref{kutsig1}\\
   $\mu_2h_{\min}/(1+\mu_2^{1/2}r_{\max})\leq 2\sigma_2$  &star-shaped with respect to a point&Thm. \ref{kutsig2} \\
 $\eta_1\leq\frac{1}{2}\lambda_1h_{\max}$    &star-shaped with respect to a point&Thm. \ref{kutsig3} $(i)$\\
     $\lambda_1^{1/2}\leq 2\eta_1r_{\max}/h_{\min}$  &star-shaped with respect to a point &Thm. \ref{kutsig3} $(i)$\\
 $\xi_2\leq\mu_2^2h_{\max}$    & star-shaped with respect to its centroid&Thm. \ref{kutsig3} $(ii)$\\

 \hline
\end{tabular}
\end{table}
\end{center}
We  extend Kuttler--Sigillito's results in two ways. 
Firstly, we consider domains $\Omega$ with $C^2$ boundary in a complete Riemannian manifolds $(M,g)$ of arbitrary dimension $n\geq 2$.
Secondly, we also prove inequalities between higher-order eigenvalues. 

\smallskip

Our first theorem provides lower bounds for $\xi_k$ in terms of Neumann and Steklov eigenvalues.

\begin{theorem}\label{kutsig1}
For every $k\in \N$ we have 
(a) $\mu_k\sigma_2\le\xi_k$, and
(b)
$\mu_2\sigma_k\le\xi_k
$.
\end{theorem}
Compared to inequality~(b), inequality~(a) gives a better lower bound for $\xi_k$ for large $k$.  
For $k=2$ and $\Omega\subset\R^2$, Theorem\,\ref{kutsig1} was previously proved in \cite{KS68}.  Kuttler in \cite{Kut82} also obtained an inequality between some higher order eigenvalues $\xi_k$ and  $\mu_k$ for a rectangular domain in $\R^2$ using symmetries of the eigenfunctions. 
\smallskip

In order to state our next results, we need to introduce some notation first.
 For any given $p\in M,$ consider the  distance function $$d_p:\Omega\to[0,\infty),\quad d_p(x):=d(p,x),$$ and  one half of the square of the distance function,  $$\rho_p(x):=\frac{1}{2}d_p(x)^2.$$  Furthermore, we set $$r_{\max}:=\max_{x\in\Omega}d_p(x)=\max_{x\in\partial\Omega}d_p(x),$$
$$h_{\max}:=\max_{x\in\partial\Omega}\la\nabla\rho_p,\nu\ra,\quad\text{and}\quad h_{\min}:=\min_{x\in\partial\Omega}\la\nabla\rho_p,\nu\ra,$$
where we borrowed the notation from \cite{KS68}.

\smallskip

%
%
We shall see that under the assumption of a lower Ricci curvature bound, there exists a lower bound on the first nonzero Steklov eigenvalue $\sigma_2$ in terms of  $\mu_2$ on star shaped domains.
\begin{theorem}\label{kutsig2}
 Let the Ricci curvature ${\rm Ric}_g$ of the ambient space $M$ be bounded from below  $${\rm Ric}_g\ge(n-1)\kappa,$$ and let $\Omega\subset M$ be a bounded star shaped domain with respect to $p\in\Omega$. Then we have
\begin{align}\label{sigmamu}
\sigma_2\ge \frac{h_{\min}\mu_2}{2r_{\max}\mu_2^{1/2}+C_0}, 
\end{align}
where $C_0:=C_0(n,\kappa, r_{\max})$ is a positive constant depending only on $n,\kappa$ and $r_{\max}$.
\end{theorem}
\noindent When the ambient space $M$ is Euclidean, inequality \eqref{sigmamu} was stated in \cite{KS68} with $C_0=2$.  \\

In the following theorem we provide several inequalities for eigenvalues of \eqref{neu}--\eqref{bihar-ii} on star shaped domains under the assumption of bounded sectional curvature.
Here and hereafter, we make use of the notation $$A\vee B:=\max\{A,B\}\qquad\mbox{for all}\,\,A,B\in \R,$$ and the
convention $c/0=+\infty$, $c\in\R\smallsetminus\{0\}$.  

\smallskip

\begin{theorem}\label{kutsig3} 
Let the sectional curvature $K_g$  of the ambient space $M$ satisfy $\kappa_1\le K_g\le\kappa_2.$
Moreover, let $\Omega\subset M$ be a star shaped domain with respect to $p\in \Omega$ which is contained in the complement of the cut locus of $p$. Then there exist constants $C_i:=C_i(n,\kappa_1,\kappa_2,r_{\max})$, i=1,2, depending only on $n$, $\kappa_1,\kappa_2$ and  $r_{\max}$ and $C_3=C_3(n,\kappa_1,r_{\max})$ such that\smallskip

\begin{itemize}
\setlength\itemsep{1em}
\item[i)]${C_1\eta_m}/{h_{\max}\le \lambda_k}\le\left(4r_{\max}^2\eta_k^2-2C_2 h_{\min}\eta_k\right)/{h_{\min}^2},$
where $m$ is the multiplicity of $\lambda_k$;
\item[ii)] 
$\xi_{m+1}\le
{h_{\max}\mu_k^2}/\left((C_3- n^{-1}\vol(\Omega)^{-1}\mu_k\int_\Omega d_p^2\,dv_g)\vee0\right)$,
{provided $\kappa_2\le0$}.
\end{itemize}
\end{theorem}
Note that the constants $C_i$, $i=1,2,3$ are not positive in general. However, there exists $r_0:=r_0(n,\kappa_1,\kappa_2)>0$ such that for $r_{\max}\le r_0$ these constants are positive; see Section\,\ref{KSinequalities} for details. In inequality $ii)$, we have a non trivial upper bound only if $$\mu_k< nC_3 \vol(\Omega)\left(\int_\Omega d_p^2\,dv_g\right)^{-1}.$$  When $\Omega$ is a domain in $\R^n$, the quantity $\int_\Omega d_p^2 dv_g$ is called the second moment of inertia; see Example \ref{rncase}.
The proof of Theorem \ref{kutsig3} also leads to a non-sharp lower bound on $\eta_1$ $$\eta_1\ge\frac{h_{\min}C_2}{r^2_{\max}}.$$ 
This in particular shows that the right-hand side of the inequality in part $i)$ is always positive.

\smallskip

The proof of Theorem \ref{kutsig1} is based on using the variational characterization of the eigenvalues and alternative formulations thereof.  Apart from the Laplace and Hessian comparison theorems, and the variational characterization of the eigenvalues, the key tool in the proof of Theorems\,\ref{kutsig2} and \ref{kutsig3} is a generalization of the classical Rellich identity to the manifold setting.  This is the content of the next theorem.

\begin{theorem}[Generalized Rellich identity]\label{nr}
Let $F:\Omega\to T\Omega$ be a Lipschitz vector field on $\Omega$. Then for every  $w\in C^2(\Omega)$ we have
\begin{multline*}
\int_{\Omega}(\Delta w+\lambda w) \langle F,\nabla w\rangle dv_g
=\int_{\partial\Omega}{\partial_\nu w}\langle F,\nabla w\rangle ds_g-\frac{1}{2}\int_{\partial\Omega}\vert\nabla w\vert^2 \langle F,\nu\rangle ds_g+\frac{\lambda}{2}\int_{\partial \Omega}w^2 \langle F,\nu \rangle ds_g\\
+\frac{1}{2}\int_{\Omega}\diver F\lvert \nabla w\lvert^2dv_g-\int_{\Omega}DF(\nabla w,\nabla w)dv_g
-\frac{\lambda}{2}\int_{\Omega}w^2\diver F\, dv_g,
\end{multline*}
where $\nu$ denotes the outward pointing normal and $\langle \,\cdot\,,\,\cdot\, \rangle=g( \,\cdot\,,\,\cdot\,)$.
\end{theorem}The classical Rellich identity was first stated by Rellich in \cite{R40}.  A special case  of Theorem \ref{nrellich},  called the generalized Pohozaev identity, was proved in  \cite{PS16,Xio17} in order to get some spectral inequalities between the Steklov and Laplace eigenvalues. 
\bigskip
  
 The paper is structured as follows.
 In Section\,\ref{prelim}, we recall tools 
 needed in later sections, namely the Hessian and Laplace comparison theorems. 
  Moreover, we give variational characterizations and alternative representations for the eigenvalues of  problems \eqref{neu}--\eqref{bihar-ii}.
Section\,\ref{rellich} contains the deduction of the Rellich identity on manifolds, as well as several applications thereof.
 Finally, we prove the main theorems  in Section\,\ref{KSinequalities}.

\begin{acknowledge} 
The authors are grateful to Werner Ballmann and Henrik Matthiesen for valuable suggestions.
Furthermore, the authors would like to thank
the Max Planck Institute for Mathematics in Bonn (MPIM) for supporting a research visit of the first named author. This work was mainly completed when the first named author was an EPDI postdoctoral fellow  at the Mittag-Leffler Institute. The first named author would like to thank the Mittag-Leffler Institute and 
the second named author would like to thank the MPIM for the support and for providing excellent working conditions.
\end{acknowledge}
\section{Preliminaries}

\label{prelim}
In this section we provide the basic tools needed in later sections.
Namely, we give the variational characterizations and alternative representations of the eigenvalues of problems~\eqref{neu}-\eqref{bihar-ii}  in the first subsection.  In the second subsection, we recall  the Hessian and Laplace comparison theorems.

\subsection{Variational characterization and alternative representations}
\label{vc}
Below, we list the variational characterization of eigenvalues of \eqref{neu}--\eqref{bihar-ii} and their alternative representations. 
For the special case $\Omega\subset\R^2$, the proofs are contained in \cite{KS68}.
The general proofs follow along the lines of these proofs and are therefore omitted.

\smallskip

\noindent \textit{Dirichlet eigenvalues:}
\bqn{dirvar}
\lambda_k&=&\inf_{\substack{ V\subset H^1_0(\Omega)\\ \dim V=k}}\sup_{0\neq u\in V}\frac{\int_\Omega |\nabla u|^2\,dv_g}{\int_{\Omega}u^2\,dv_g}\\
\nonumber&=&\inf_{\substack{ V\subset H^2(\Omega)\cap H^1_0(\Omega)\\\dim V=k}}\sup_{0\neq u\in V}\frac{\int_\Omega (\Delta u)^2\,dv_g}{\int_{\Omega}|\nabla u|^2\,dv_g}.
\eqn
\textit{Neumann eigenvalues:}
\bqn{neuvar}
\mu_k&=&\inf_{\substack{ V\subset H^1(\Omega)\\ \dim V=k}}\sup_{0\neq u\in V}\frac{\int_\Omega |\nabla u|^2\,dv_g}{\int_{\Omega}u^2\,dv_g}
\\
\nonumber&=&\inf_{\substack{ V\subset H^2(\Omega)\\  {\partial_\nu u}=0\;\text{on $\partial \Omega$}\\\dim V=k}}\sup_{0\neq u\in V}\frac{\int_\Omega (\Delta u)^2\,dv_g}{\int_{\Omega}|\nabla u|^2\,dv_g}.
\eqn
\textit{Steklov eigenvalues:}
\bqn{stekvar}
\sigma_k&=&\inf_{\substack{ V\subset H^1(\Omega)\\ \dim V=k}}\sup_{0\neq u\in V}\frac{\int_\Omega |\nabla u|^2\,dv_g}{\int_{\partial\Omega}u^2\,dv_g}
\\
\nonumber&=&\inf_{\substack{ V\subset \cH(\Omega)\\\dim V=k}}\sup_{0\neq u\in V}\frac{\int_{\partial\Omega}({\partial_\nu u})^2\,ds_g}{\int_{\Omega}|\nabla u|^2\,dv_g},
\eqn
where $\cH(\Omega)$ is the space of harmonic functions on $\Omega$.

\medskip

\noindent \textit{Biharmonic Steklov I eigenvalues:}
\ben{bihar1var}
\eta_k=\inf_{\substack{ V\subset H^2(\Omega)\cap H^1_0(\Omega)\\\dim V=k}}\sup_{\substack{0\neq u\in V}}\frac{\int_\Omega |\Delta u|^2\,dv_g}{\int_{\partial\Omega}({\partial_\nu u})^2\,ds_g}.
\een

\noindent \textit{Biharmonic Steklov II eigenvalues:}

\ben{bihar2var}
\xi_k=\inf_{\substack{ V\subset H_N^2(\Omega)\\\dim V=k}}\sup_{\substack{0\neq u\in V}}\frac{\int_\Omega |\Delta u|^2\,dv_g}{\int_{\partial\Omega}u^2\,ds_g},
\een
where $H_N^2(\Omega):=\{u\in H^2(\Omega)\,:\, {\partial_\nu u}=0\,\,\text{on $\partial \Omega$}\}$.

\subsection{Hessian and Laplace comparison theorems}
The idea of comparison theorems is to compare a given geometric quantity on a Riemannian manifold with the corresponding quantity on a model space. Below we recall the Hessian and Laplace comparison theorems.
For more details we refer the reader to \cite{BM03,C06,P06} and \cite{C06,P06}, respectively.

\smallskip

For any $\kappa\in \R$, denote by $H_\kappa:[0,\infty)\to\R$ the function satisfying the Riccati equation 
\[H_\kappa'+H_\kappa^2+\kappa=0,\quad\mbox{with}\quad \lim_{r\to0}\frac{rH_\kappa(r)}{n-1}=1.\]
Clearly, we have
$$H_\kappa(r)=\begin{cases}
(n-1)\sqrt{\kappa}\cot(\sqrt{\kappa}r)&\kappa>0,\\
\frac{n-1}{r}&\kappa=0,\\
(n-1)\sqrt{|\kappa|}\coth(\sqrt{|\kappa|}r)&\kappa<0.
\end{cases}
$$

With this preparation at hand we can now state the Hessian comparison theorem.

\begin{theorem}[Hessian comparison theorem]\label{hessian} Let $\gamma:[0,L]\to M$ be a minimizing geodesic starting from $p\in M$, such that its image is disjoint to the cut locus of $p$. 
Assume furthermore that $$\kappa_1\le K_g(X,\dot\gamma(t))\le \kappa_2$$ for all $t\in[0,L]$ and $X\in T_{\gamma(t)}M$ perpendicular to $\dot\gamma(t)$. Then
 \begin{itemize}\item[(a)]
 $d_p$ satisfies the inequalities
 \begin{align*}
& \nabla^2d_p(X,X)\le \frac{H_{\kappa_1}(t)}{n-1}g(X,X),\qquad \forall t\in[0,L],\quad X\in \langle\dot\gamma(t)\rangle ^\perp\subset T_{\gamma(t)}M,\\
&\nabla^2d_p(X,X)\ge \frac{H_{\kappa_2}(t)}{n-1}g(X,X),\qquad\forall t\in[0,L\wedge\frac{\pi}{2\sqrt{\kappa_2\vee0}}], \quad X\in \langle\dot\gamma(t)\rangle ^\perp\subset T_{\gamma(t)}M.
\end{align*}
Furthermore, we have
\[\nabla^2d_p(\dot\gamma(t),\dot\gamma(t))=0,\qquad \forall t\in[0,L].\]
Here $A\wedge B:=\min\{A,B\}$  and $A\vee B:=\max\{A,B\}$ for $A,B\in \R$.
\item[(b)]
$\rho_p$ satisfies the inequalities
\begin{align*}
&\nabla^2\rho_p(X,X)\le \frac{tH_{\kappa_1}(t)}{n-1}g(X,X),\qquad \forall t\in[0,L],\quad X\in \langle\dot\gamma(t)\rangle ^\perp\subset T_{\gamma(t)}M,\\
&\nabla^2\rho_p(X,X)\ge \frac{tH_{\kappa_2}(t)}{n-1}g(X,X),\qquad\forall t\in[0,L\wedge\frac{\pi}{2\sqrt{\kappa_2\vee0}}], \quad X\in \langle\dot\gamma(t)\rangle^\perp\subset T_{\gamma(t)}M,
\end{align*}
and 
\[\nabla^2\rho_p(\dot\gamma(t),\dot\gamma(t))=1,\qquad \forall t\in[0,L].\]
\end{itemize}
\end{theorem}   
Next, we state the Laplace comparison theorem. 

\begin{theorem}[Laplace comparison theorem]\label{laplacecomp}
The distance function $d_p$ and the squared distance function satisfy the following.
 \begin{itemize}\item[(a)]Let $\Ric_g\ge(n-1)\kappa$, $\kappa\in \R$. Then for every $p\in M$ the inequalities 
\[\Delta d_p(x)\le H_\kappa(d_p(x)),\quad\mbox{and}\quad
\Delta \rho_p(x)\le 1+d_p(x)H_\kappa(d_p(x))\]
hold at smooth points of $d_p$. Moreover the above  inequalities hold on the whole manifold in the sense of distribution.
\item[(b)] Under the same assumption and notations of  Theorem \ref{hessian}, the following inequalities hold. 
\begin{itemize}
\item[(i)] For every  $t\in[0,L]$\[\Delta d_p(\gamma(t))\le H_{\kappa_1}(t),\quad\mbox{and}\quad
\Delta \rho_p(\gamma(t))\le 1+tH_{\kappa_1}(t)\,;\]
\item [(ii)]For every $ t\in[0,L\wedge\frac{\pi}{2\sqrt{\kappa_2\vee0}}]$ \[\Delta d_p(\gamma(t))\ge H_{\kappa_2}(t),\quad\mbox{and}\quad\Delta \rho_p(\gamma(t))\ge 1+tH_{\kappa_2}(t).\]
\end{itemize}
\end{itemize}
\end{theorem}   
Notice that part $(b)$ in the above theorems is an immediate consequence of part $(a)$, since 
the distance function $d_p$ and one half of the square of the distance function $\rho_p$ satisfy
\[\nabla^2\rho_p=d_p\nabla^2d_p+\nabla d_p\otimes\nabla d_p,\qquad\Delta \rho_p=|\nabla d_p|^2+d_p\Delta d_p.\]

\section{Generalized Rellich identity}
\label{rellich}
An important identity which is used in the study of eigenvalue problems is the Rellich identity. To our knowledge it was first stated and used by Rellich \cite{R40} in the study of the eigenvalue problem. 
Some versions of the Rellich identity are also referred to as Pohozaev identity; see \cite{PS16,Xio17}. 
In this section, we provide the generalized Rellich identity on Riemannian manifolds, i.e. Theorem\,\ref{nr}, and its higher order version. Applications of this result can be found in the last subsection and in  Section\,\ref{KSinequalities}.

\subsection{Rellich identity on manifolds}
The next theorem states the Rellich identity on Riemannian manifolds. 
\begin{theorem}[Generalized Rellich identity for manifolds]\label{nrellich}
Let $(\Omega,g)$ be a Riemannian manifold with piecewise smooth boundary. Let $F:\Omega\to T\Omega$ be a Lipschitz vector field on $\Omega$. Then for every  $w\in C^2(\Omega)$ we have
\begin{multline*}
\int_{\Omega}(\Delta w+\lambda w) \langle F,\nabla w\rangle dv_g
=\int_{\partial\Omega}{\partial_\nu w}\langle F,\nabla w\rangle ds_g-\frac{1}{2}\int_{\partial\Omega}\vert\nabla w\vert^2 \langle F,\nu\rangle ds_g+\frac{\lambda}{2}\int_{\partial \Omega}w^2 \langle F,\nu \rangle ds_g\\
+\frac{1}{2}\int_{\Omega}\diver F\lvert \nabla w\lvert^2dv_g-\int_{\Omega}DF(\nabla w,\nabla w)dv_g
-\frac{\lambda}{2}\int_{\Omega}w^2\diver F\, dv_g,
\end{multline*}
where $\nu$ denotes the outward pointing normal and $\langle \,\cdot\,,\,\cdot\, \rangle=g( \,\cdot\,,\,\cdot\,)$.
\end{theorem}
In \cite{PS16,Xio17}, the authors proved the above identity when $w$ is harmonic and $\lambda=0$. The proof of the general version follows the same line of argument.  For the sake of completeness we give the whole argument. 
\begin{proof}[Proof of Theorem \ref{nrellich}]
 We calculate $\int_{\Omega}\Delta w\langle F,\nabla w\rangle dv_g$ and $\int_{\Omega}\lambda w\langle F,\nabla w\rangle dv_g$ separately. 
In order to calculate the latter, we apply the divergence theorem to obtain
\begin{align*}
\int_{\partial\Omega} w^2\la F,\nu\ra\,ds_g=\int_\Omega\diver (w^2F)dv_g=\int_\Omega\left(2 w \langle F,\nabla w\rangle+w^2 \diver F\right)dv_g.
\end{align*}
Thus, we get
\begin{align*}
\int_{\Omega}\lambda w\langle F,\nabla w\rangle dv_g=\frac{\lambda}{2}\left(\int_{\partial \Omega}w^2 \langle F,\nu \rangle ds_g-\int_{\Omega} w^2\diver F\,dv_g\right).
\end{align*}
For the other term, using integration by parts, we obtain 
\begin{multline}
\label{lhs}
\int_{\Omega}\Delta w\langle F,\nabla w\rangle dv_g=
\int_{\partial\Omega}\langle F,\nabla w\rangle{\partial_\nu w} ds_g-\int_{\Omega}\langle \nabla\langle F,\nabla w\rangle,\nabla w\rangle dv_g\\=
\int_{\partial\Omega}\langle F,\nabla w\rangle{\partial_\nu w} ds_g-\int_{\Omega}\langle \nabla_{\nabla w}F,\nabla w\rangle dv_g-\int_{\Omega}\langle \nabla_{\nabla w}\nabla w,F\rangle dv_g\\
=\int_{\partial\Omega}\langle F,\nabla w\rangle{\partial_\nu w} ds_g-\int_{\Omega}DF(\nabla w,\nabla w) dv_g-\int_{\Omega}\nabla^2w(\nabla w,F) dv_g.
\end{multline}
For further simplification, we observe that
\begin{multline*}
2\int_\Omega\nabla^2w(\nabla w,F)dv_g=\int_\Omega\diver(F\lvert\nabla w\lvert^2)dv_g-\int_\Omega\diver F\lvert\nabla w\lvert^2dv_g\\=\int_{\partial\Omega}\lvert\nabla w\lvert^2Fds_g-\int_\Omega\diver F\lvert\nabla w\lvert^2dv_g.
\end{multline*}
Plugging this identity into (\ref{lhs})  we get
\begin{align*}
\int_{\Omega}\Delta w\langle F,\nabla w\rangle dv_g=&\int_{\partial\Omega}{\partial_\nu w}\langle F,\nabla w\rangle ds_g-\frac{1}{2}\int_{\partial\Omega}\vert\nabla w\vert^2 \langle F,\nu\rangle ds_g\\&+\frac{1}{2}\int_{\Omega}\diver F\lvert \nabla w\lvert^2dv_g-\int_{\Omega}DF(\nabla w,\nabla w)dv_g.\end{align*}
This completes the proof.
\end{proof}
\subsection{Higher order Rellich identities}
In this section we provide a higher order Rellich identity.

\smallskip

The following preparatory lemma is a simple consequence from Theorem\,\ref{nrellich}.
For the special case $M=\R^n$,
the identity stated in the lemma was first proven by Mitidieri in \cite{M93}.

\begin{lemma}
\label{polar}
For $u,v\in C^2(\Omega)$ we have 
\begin{multline*}
\int_{\Omega}\Delta w \langle F,\nabla v\rangle +\Delta v \langle F,\nabla w\rangle dv_g
=\int_{\partial\Omega}\{{\partial_\nu w}\langle F,\nabla v\rangle +{\partial_\nu v}\langle F,\nabla w\rangle \}ds_g
-\int_{\partial\Omega} \langle\nabla w,\nabla v \rangle \langle F,\nu\rangle ds_g\\
+\int_{\Omega}\diver F\langle\nabla w,\nabla v \rangle dv_g
-2\int_{\Omega}DF(\nabla w,\nabla v)dv_g.
\end{multline*}
\end{lemma}
\begin{proof}
Replacing $w$ by $w+v$ in Theorem\,\ref{nrellich} and set $\lambda=0$ we get the identity.
\end{proof}

The following theorem states the higher order Rellich identity.

\begin{theorem}
\label{Rellich-2nd}
For $w\in C^4(\Omega)$ we have
\begin{multline*}
\int_{\Omega}(\Delta^2w+\lambda \Delta  w) \langle F,\nabla w\rangle dv_g
=\frac{1}{2}\int_{\Omega}\diver  F(\Delta w)^2 dv_g-\frac{1}{2}\int_{\partial\Omega}(\Delta w)^2 \langle F,\nu\rangle dv_g\\
+\int_{\partial\Omega}\{{\partial_\nu w}\langle F,\nabla \Delta  w\rangle +{\partial_\nu \Delta  w}\langle F,\nabla w\rangle \}ds_g-\int_{\partial\Omega} \langle\nabla w,\nabla \Delta  w \rangle \langle F,\nu\rangle ds_g\\
+\int_{\Omega}\diver  F\langle\nabla w,\nabla \Delta  w \rangle dv_g
-2\int_{\Omega}DF(\nabla w,\nabla \Delta  w)dv_g+
\lambda \int_{\partial\Omega}{\partial_\nu w}\langle F,\nabla w\rangle ds_g\\-\frac{\lambda}{2}\int_{\partial\Omega}\vert\nabla w\vert^2 \langle F,\nu\rangle ds_g+\frac{\lambda}{2}\int_{\Omega}\diver F\lvert \nabla w\lvert^2dv_g-\lambda\int_{\Omega}DF(\nabla w,\nabla w)dv_g.
\end{multline*}
\end{theorem}
\begin{proof}
If we choose $v=\Delta  w$ in Lemma\,\ref{polar}, we obtain 
\begin{eqnarray*}
\int_{\Omega}\Delta^2w \langle F,\nabla w\rangle dv_g
&=&-\int_{\Omega}\Delta w \langle F,\nabla \Delta  w\rangle dv_g\\&&+\int_{\partial\Omega}\{{\partial_\nu w}\langle F,\nabla \Delta  w\rangle +{\partial_\nu \Delta  w}\langle F,\nabla w\rangle \}ds_g
-\int_{\partial\Omega} \langle\nabla w,\nabla \Delta  w \rangle \langle F,\nu\rangle ds_g\\
&&+\int_{\Omega}\diver  F\langle\nabla w,\nabla \Delta  w \rangle dv_g
-2\int_{\Omega}DF(\nabla w,\nabla \Delta  w)dv_g.
\end{eqnarray*}
By the divergence theorem we have
\begin{align*}
\int_{\Omega}\Delta w \langle F,\nabla \Delta  w\rangle dv_g&=\frac{1}{2}\int_{\Omega}\la F,\nabla(\Delta w)^2 \ra dv_g
\\&=-\frac{1}{2}\int_{\Omega}\diver  F(\Delta w)^2 dv_g+\frac{1}{2}\int_{\partial\Omega}(\Delta w)^2 \langle F,\nu\rangle dv_g,
\end{align*}
which together with Theorem\,\ref{nrellich} establishes the claim.
\end{proof}

For the special case $M=\R^n$ and $\lambda=0$, the statement of Theorem\,\ref{Rellich-2nd} is contained in \cite{M93}.

\subsection{Applications of the Rellich identities}

In 1940 Rellich \cite{R40} dealt with the Dirichlet eigenvalue problem on sets $\Omega\subset\R^n$. For this special case he used  the identity derived in Theorem\,\ref{nrellich} to express the Dirichlet eigenvalues in terms of an integral over the boundary.
One decade ago, Liu~\cite{Liu07} extended Rellich's result to the Neumann eigenvalue problem, the clamped plate eigenvalue problem and the buckling eigenvalue problem, each on sets $\Omega\subset\R^n$. In the latter two cases Liu (implicitly) applied the higher order Rellich identity.      

\smallskip

Recall that for any bounded domain $\Omega\subset M$ with $C^2$ boundary $\partial \Omega$
the clamped plate eigenvalue problem and the buckling eigenvalue problem are given by 
 \begin{align}
 & \label{buck} \left\{ \begin{array}{ll} \Delta^2 u+\Lambda
  \Delta u=0 \quad \;\; &\mbox{in}\;\; \Omega,\\
  u={\partial_\nu u}=0 \quad \;\;  &  \mbox{on}\;\;
   \partial \Omega; \end{array} \right. &\qquad\qquad \mbox{Buckling problem\,,}\\
   & \label{clamp}   \left\{ \begin{array}{ll} \Delta^2 u
  -\Gamma^2 u=0 \quad \;\; &\mbox{in}\;\; \Omega,\\
  u={\partial_\nu u}=0\quad \;\;  &  \mbox{on}\;\;
  \partial \Omega; \end{array} \right.  &\qquad\qquad \mbox{Clamped plate\,,}&
       \end{align}
     respectively.

Below we reprove the result of Liu for the case of the buckling eigenvalue problem.
Note there is no new idea for the proof, however, our proof is shorter and clearer since we do not carry out the calculations in coordinates. One can proceed similarly for the clamped plate eigenvalue problem.

\smallskip

\begin{lemma}[\cite{Liu07}]\label{rell}
Let $\Omega\subset\R^n$ be a bounded domain with smooth boundary.
\begin{itemize}
\item[(i)]
Let $w$ be an eigenfunction corresponding to the eigenvalue $\Lambda$ of the buckling eigenvalue problem.
Then we have
\begin{align*}
\Lambda=\frac{\int_{\partial\Omega}({\partial_{\nu\nu}^2w})^2{\partial_\nu (r^2)}ds_g}{4\int_{\Omega}\lvert\nabla w\rvert^2dv_g},
\end{align*}
where $r^2=x_1^2+\dots+x_n^2$ and $x_i$ are Euclidean coordinates.
\item[(ii)]
Let $w$ be an eigenfunction corresponding to the eigenvalue $\Gamma$ of the clamped plate eigenvalue problem.
Then we have
\begin{align*}
\Gamma=\frac{\int_{\partial\Omega}({\partial_{\nu\nu}^2w})^2{\partial_\nu (r^2)}ds_g}{8\int_{\Omega}w^2dv_g}.
\end{align*}
\end{itemize}
\end{lemma}
\begin{proof}
In order to prove $(i)$ we apply Theorem\,\ref{Rellich-2nd} for the special case $\Omega\subset\R^n$ and where $F$ is given by the gradient of the distance function. In this case we have $DF(\,\cdot\, ,\,\cdot\,)=g(\,\cdot\, ,\,\cdot\,)$ and $\diver F=n$.
Note furthermore that $w_{\lvert\partial\Omega}=0$ implies
$\nabla w={\partial_\nu w}\nu$ on $\partial\Omega$.
Since we have ${\partial_\nu w}_{\lvert\partial\Omega}=0$ by assumption, $\nabla w$ vanishes along the boundary of $\Omega$.

\smallskip

Plugging the above information into the higher order Rellich identity we get
\begin{align*}
0=\int_{\Omega}(\Delta^2w+\lambda \Delta  w) \langle F,\nabla w\rangle dv_g
=\frac{n}{2}\int_{\Omega}(\Delta w)^2 dv_g-\frac{1}{2}\int_{\partial\Omega}(\Delta w)^2 \langle F,\nu\rangle dv_g\\
+(n-2)\int_{\Omega}\langle\nabla w,\nabla \Delta  w \rangle dv_g+\Lambda(\frac{n}{2}-1)\int_{\Omega}\lvert \nabla w\lvert^2dv_g.
\end{align*}
Applying the divergence theorem once more, we thus obtain
\begin{align*}
\Lambda(\frac{n}{2}-1)\int_{\Omega}\lvert\nabla w\rvert^2dv_g=\frac{1}{2}\int_{\partial\Omega}(\Delta w)^2\langle F,\nu\rangle ds_g-(2-\frac{n}{2})\int_{\Omega}(\Delta w)^2dv_g.
\end{align*}
The variational characterization of $\Lambda$ asserts that for an eigenfunction $w$ corresponding to $\Lambda$ we have
\begin{align}
\label{var-buck}
\int_{\Omega}(\Delta w)^2dv_g-\Lambda\int_{\Omega}\lvert\nabla w\lvert^2dv_g=0.
\end{align}
Furthermore, the identities
$$\langle F,\nu\rangle=\sum_{i=1}^nx_i{\partial_\nu x_i}=\frac{1}{2}{\partial_\nu (r^2)}$$
and $\Delta w={\partial_{\nu\nu}^2w}$ hold on the boundary of $\Omega$. Thus the claim is established.

\smallskip

The proof of $(ii)$ is omitted since it is similar to the one of $(i)$.
\end{proof}

\begin{rem}
In Lemma\,\ref{rell} $(i)$, when normalizing the eigenfunction $w$ such that $\int_{\Omega}\lvert\nabla w\rvert^2dv_g=1$, we obtain
\begin{align*}
\Lambda=\frac{1}{4}\int_{\partial\Omega}({\partial_{\nu\nu}^2w})^2{\partial_\nu (r^2)}ds_g;
\end{align*}
i.e. $\Lambda$ is expressed in terms of an integral over the boundary.  A similar remark holds for Lemma\,\ref{rell}~$(ii)$.
\end{rem}

\bigskip

Finally we use the Rellich identities to get some estimates on
eigenvalues.
Note that from now on we do not assume anymore that $\Omega$ is a subset of the Euclidean space.
However, we assume that $\Omega$ is a manifold with smooth boundary and that there exists a vector field $F$ on $\Omega$ satisfying the following properties:
\begin{itemize}
\item[A)] $0<c_1\le\diver F\le c_2$, for some positive constants $c_1,c_2\in \R_+$,
\item[B)] $DF(X,X)\ge \alpha g(X,X)$ for some positive constant $\alpha \in \R_+$,\label{i-iii}
\item[C)] $\la F,\nu \ra\ge0$ on $\partial\Omega$. 
\end{itemize}

\begin{rem}
Domains in Hadamard manifolds, and free boundary minimal hypersufaces in the unit ball in $\R^{n+1}$ provide examples for which conditions A-C for the gradient of the distance function on $\Omega$ are satisfied.  For the latter see Example   \ref{minsurf} in which condition A with  $c_1=c_2$ holds.
\end{rem}

The following lemma is an easy consequence of Theorem\,\ref{nrellich} and Theorem\,\ref{Rellich-2nd}, respectively. 
It establishes upper estimates for eigenvalues in terms of integrals over the boundary $\partial \Omega$ and $\alpha$.

\begin{lemma}\label{3.10}
Assume that there exists a vector field $F$ on $\Omega\subset M^n$ satisfying properties~A-C above.
Then
\begin{itemize}
\item[(i)] the eigenvalue $\lambda$ corresponding to eigenfunction $w$ of the Dirichlet eigenvalue problem satisfies
\begin{align*}
\lambda\leq\frac{\int_{\partial\Omega}({\partial_\nu w})^2\langle F,\nu\rangle ds_g}{(2\alpha+c_1-c_2)\int_{\Omega w^2dv_g}};
\end{align*}
\item[(ii)]
the eigenvalue $\Lambda$ corresponding to eigenfunction $w$ of the buckling eigenvalue problem satisfies
\begin{align*}
\frac{\int_{\partial\Omega}(\Delta w)^2 \langle F,\nu\rangle dv_g}{2\alpha\int_{\Omega}\lvert \nabla w\lvert^2dv_g}\leq\Lambda
\end{align*}
provided $c_1=c_2=:c$ in property A.
\end{itemize}
\end{lemma}

\begin{proof}
We start by proving ${(i)}$. Theorem\,\ref{nrellich} implies
\begin{multline*}0=
\int_{\Omega}(\Delta w+\lambda w) \langle F,\nabla w\rangle dv_g
\le\int_{\partial\Omega}{\partial_\nu w}\langle F,\nabla w\rangle ds_g-\frac{1}{2}\int_{\partial\Omega}\vert\nabla w\vert^2 \langle F,\nu\rangle ds_g\\
+\frac{c_2}{2}\int_{\Omega}\lvert \nabla w\lvert^2dv_g-\int_{\Omega}DF(\nabla w,\nabla w)dv_g
-\frac{\lambda c_1}{2}\int_{\Omega}w^2\, dv_g.
\end{multline*}
Since $w\equiv 0$ on $\partial\Omega$ we have $\nabla w={\partial_\nu w}\nu$ on $\partial\Omega$.
Thus, we obtain
\begin{align*}
\frac{\lambda c_1}{2}\int_{\Omega}w^2 dv_g\leq
\frac{1}{2}\int_{\partial\Omega}({\partial_\nu w})^2 \langle F,\nu\rangle ds_g
+(\frac{\lambda c_2}{2}-\alpha\lambda)\int_{\Omega}w^2\, dv_g.
\end{align*}
The latter inequality implies the claim.

\smallskip

Below we prove $(ii)$. Theorem\,\ref{Rellich-2nd} implies
\begin{multline*}
0
=\frac{c}{2}\int_{\Omega}(\Delta w)^2 dv_g-\frac{1}{2}\int_{\partial\Omega}(\Delta w)^2 \langle F,\nu\rangle dv_g+c\int_{\Omega}\langle\nabla w,\nabla \Delta  w \rangle dv_g\\
-2\int_{\Omega}DF(\nabla w,\nabla \Delta  w)dv_g+\frac{c\Lambda}{2}\int_{\Omega}\lvert \nabla w\lvert^2dv_g-\Lambda\int_{\Omega}DF(\nabla w,\nabla w)dv_g\\\leq
(2\alpha-\frac{c}{2})\int_{\Omega}(\Delta w)^2 dv_g-\frac{1}{2}\int_{\partial\Omega}(\Delta w)^2 \langle F,\nu\rangle dv_g+(\frac{c\Lambda}{2}-\Lambda\alpha)\int_{\Omega}\lvert \nabla w\lvert^2dv_g,
\end{multline*}
where we made use of 
$$\int_{\Omega}\langle\nabla w,\nabla \Delta  w \rangle dv_g=-\int_{\Omega}(\Delta  w)^2dv_g,$$
what is a consequence of the divergence theorem.
Applying (\ref{var-buck}) yields
\begin{align*}
0\leq -\frac{1}{2}\int_{\partial\Omega}(\Delta w)^2 \langle F,\nu\rangle dv_g+\Lambda\alpha\int_{\Omega}\lvert \nabla w\lvert^2dv_g,
\end{align*}
and thus the claim is established.
\end{proof}

\section{Proof of the Main Theorems}
\label{KSinequalities} 

\smallskip

%
%

\bff[Proof of Theorem \ref{kutsig1}]  Inequalities $(a)$ and $(b)$ are an immediate consequence of the variational characterizations of $\mu_k$, $\sigma_k$ and $\xi_k$ given in \eqref{neuvar}, \eqref{stekvar} and \eqref{bihar2var}. Indeed, let $V$ be the space generated by  eigenfunctions associated with $\xi_2,\ldots,\xi_k$ with $\int_{\partial\Omega} u=0$, for every $u\in V$. Then by the variational characterization \eqref{neuvar} we get
\bq\mu_k&\le& \sup_{0\neq u\in V}\frac{\int_\Omega (\Delta u)^2\,dv_g}{\int_{\Omega}|\nabla u|^2\,dv_g}\le\xi_k \sup_{0\neq u\in V}\frac{\int_{\partial\Omega} u^2\,dv_g}{\int_{\Omega}|\nabla u|^2\,dv_g}\\
&=&\xi_k\left(\inf_{0\neq u\in V}\frac{\int_{\Omega}|\nabla u|^2\,dv_g}{\int_{\partial\Omega} u^2\,dv_g}\right)^{-1}\le \frac{\xi_k}{\sigma_2}.
\eq
The proof of part $(b)$ is similar and we leave it to the reader. \end{proof}
\bff[Proof of Theorem \ref{kutsig2}] 
We use  the following identity
\begin{align*}
\frac{1}{2}\int_{\partial\Omega}w^2\langle \nu,\nabla \rho_p\rangle ds_g=
\int_{\Omega}w\langle \nabla w,\nabla \rho_p\rangle dv_g+\frac{1}{2}\int_{\Omega}w^2\Delta\rho_p dv_g
\end{align*}
which follows easily from integration by parts.
Using the Laplace comparison theorem, we thus get
\begin{align}
\label{e1}
\frac{1}{2}\int_{\partial\Omega}w^2\langle \nu,\nabla \rho_p\rangle ds_g\leq
\int_{\Omega}w\langle \nabla w,\nabla \rho_p\rangle dv_g+\frac{1}{2}\max_{x\in\Omega}(1+d_p(x)H_{\kappa_1}(d_p(x)))\int_{\Omega}w^2dv_g.
\end{align}
The Cauchy Schwarz inequality yields
\begin{align*}
\big(\int_{\Omega}w\langle \nabla w,\nabla \rho_p\rangle dv_g\big)^2\leq r_{\max}^2\int_{\Omega}w^2dv_g\int_{\Omega}\lvert\nabla w\lvert^2dv_g.
\end{align*}
Assuming $\int_{\Omega}wdv_g=0$ and using the variational characterisation of $\mu_2$ we get
\begin{align*}
\int_{\Omega}w\langle \nabla w,\nabla \rho_p\rangle dv_g\leq r_{\max}\mu_2^{-1/2}\int_{\Omega}\lvert\nabla w\lvert^2dv_g.
\end{align*}
Thus, from inequality~\eqref{e1}, we get
\begin{align}
\label{e2}
\frac{1}{2}\int_{\partial\Omega}w^2\langle \nu,\nabla \rho_p\rangle ds_g\leq
(r_{\max}\mu_2^{-1/2}
+\frac{1}{2}\max_{x\in\Omega}(1+d_p(x)H_{\kappa}(d_p(x)))\mu_2^{-1})\int_{\Omega}\lvert\nabla w\lvert^2dv_g.
\end{align}
Let $u$ be an eigenfunction associated to the eigenvalue $\sigma_2$ and choose $w$ to be $$w:=u-\vol(\Omega)^{-1}\int_{\Omega}udv_g.$$ 
Then we have
\begin{align*}
\int_{\Omega}\lvert\nabla w\lvert^2 dv_g=\int_{\Omega}\lvert\nabla u\lvert^2 dv_g=\sigma_2\int_{\partial\Omega}u^2ds_g\leq\sigma_2\int_{\partial\Omega}w^2ds_g.
\end{align*}
Combining this inequality with (\ref{e2}), we finally get
\begin{align*}
\frac{1}{2}h_{\min}\int_{\partial\Omega}w^2ds_g&\leq \frac{1}{2}\int_{\partial\Omega}w^2\langle \nu,\nabla \rho_p\rangle ds_g\\&\leq (r_{\max}\mu_2^{-1/2}
+\frac{1}{2}\max_{x\in\Omega}(1+d_p(x)H_{\kappa}(d_p(x)))\mu_2^{-1})\int_{\Omega}\lvert\nabla w\lvert^2dv_g\\&\leq (r_{\max}\mu_2^{-1/2}
+\frac{1}{2}\max_{x\in\Omega}(1+d_p(x)H_{\kappa}(d_p(x)))\mu_2^{-1})\sigma_2\int_{\partial\Omega}w^2ds_g.
\end{align*}
Thus the claim is established.
\end{proof}
\bff[Proof of Theorem \ref{kutsig3}] Throughout the proof we repeatedly use the Hessian and Laplace comparison theorems as well as the generalized Rellich identity, i.e. Theorem\,\ref{nrellich}.
\begin{itemize}[leftmargin=*]
\item[\it i)] We start by proving the first inequality in $i)$, namely
${C_1\eta_m}/{h_{\max}\le \lambda_k}.$
Let $E_k$ be the eigenspace associated with $\lambda_k$ and let $u_1,\cdots,u_m$ be an orthonormal basis for $E_k$.\\ The functions ${\partial_\nu u_1}, \cdots, {\partial_\nu u_m}$ are linearly independent on $\partial\Omega$. Indeed, if there exist $u\in\Span({\partial_\nu u_1}, \cdots, {\partial_\nu u_m})$ such that ${\partial_\nu u}=0$, then we define
\[v(x)=
\begin{cases}
u(x)&\mbox{if}\,\, x\in \Omega,\\
0&\mbox{if}\,\, x\in M\setminus\Omega.
\end{cases}
\]
Clearly, we have $v\in H^1(M)$.
Furthermore, $v$ satisfies the identity $\Delta v=\lambda_kv$. Since $v\equiv0$ on $M\setminus \Omega$ we get $v\equiv0$ on $M$ by the unique continuation theorem. Thus, we can consider $E_k$ as a test functional space in \eqref{bihar1var}.\\
 Let $h_{\max}=\sup_{x\in\partial\Omega}\la\nabla\rho_p,\nu\ra$. 
Since $0<\frac{1}{h_{\max}}\la\nabla\rho_p,\nu\ra\le1$, we get 
\bq\eta_m&\le&\sup_{u\in E_k}\frac{\int_\Omega |\Delta u|^2\,dv_g}{\int_{\partial\Omega}({\partial_\nu u})^2\,ds_g}\le h_{\max}\lambda_k^2\sup_{u\in E_k}\frac{\int_\Omega u^2\,dv_g}{\int_{\partial\Omega}\la\nabla\rho_p,\nu\ra({\partial_\nu u})^2\,ds_g}.
\eq
Next we bound the denominator from below. Applying Theorem \ref{nrellich} with $\lambda=0$ and $F=\nabla\rho_p$ yields
\bq
\int_{\partial\Omega}\la\nabla\rho_p,\nu\ra({\partial_\nu u})^2\,ds_g&=&2\int_{\Omega}\Delta u\langle \nabla \rho_p,\nabla u\rangle dv_g-\int_{\Omega}\Delta \rho_p\lvert \nabla u\lvert^2dv_g+2\int_{\Omega}\nabla^2\rho_p(\nabla u,\nabla u)dv_g.
\eq
Using $u\in E_k$ and integration by parts we get
\begin{align*}
2\int_{\Omega}\Delta u\langle \nabla \rho_p,\nabla u\rangle dv_g=-\lambda_k\int_{\Omega}\langle \nabla \rho_p,\nabla u^2\rangle dv_g=\lambda_k\int_{\Omega}u^2\Delta \rho_p dv_g.
\end{align*}
Consequently, we have
\bq
\int_{\partial\Omega}\la\nabla\rho_p,\nu\ra({\partial_\nu u})^2\,ds_g&=&\lambda_k\int_{\Omega}u^2\Delta \rho_p dv_g-\int_{\Omega}\Delta \rho_p\lvert \nabla u\lvert^2dv_g+2\int_{\Omega}\nabla^2\rho_p(\nabla u,\nabla u)dv_g\\
&\ge&\lambda_k\left(1+\min_{x\in\Omega}d_p(x)H_{\kappa_2}(d_p(x))\right)\int_{\Omega}u^2 dv_g\\
&&-\left(1+\max_{x\in\Omega}d_p(x)H_{\kappa_1}(d_p(x))\right)\int_{\Omega}|\nabla u|^2 dv_g
\\&&+2 \min_{x\in\Omega}\frac{d_p(x)H_{\kappa_2}(d_p(x))}{n-1}\int_{\Omega}|\nabla u|^2 dv_g\\
&=&\lambda_kC_1\int_{\Omega}u^2 dv_g.
\eq
In the second line  we used the Hessian and Laplace comparison theorems; see Section \,\ref{prelim}.
Here $C_1$ is
\begin{equation}\label{thetaom}C_1:=\left(1+\frac{2}{n-1}\right)\min_{r\in[0,r_{\max})}rH_{\kappa_2}(r)- \max_{r\in[0,r_{\max})}rH_{\kappa_1}(r).\end{equation}
Therefore, we get
\[{C_1}\eta_m\le {h_{\max}\lambda_k}.\]
We conclude the proof of the first inequality with a remark on the sign of $C_1$. The function $rH_{\kappa}(r)$ is constant if $\kappa=0$, increasing on $[0,\infty)$ if $\kappa<0$, and decreasing on $[0,\infty)$ if $\kappa>0$. Thus we calculate $C_1$ considering the following different cases: \smallskip
\begin{itemize}
\item[(a)] If $\kappa_1=\kappa_2=0$, then $C_1=2$.
\item[(b)] If $\kappa_1\le\kappa_2\le0$, then $C_1=n+1-r_{\max}H_{\kappa_1}(r_{\max})$. 
\item[(c)] If  $0\le\kappa_1\le\kappa_2$, then $C_1=\left(1+\frac{2}{n-1}\right)r_{\max}H_{\kappa_2}(r_{\max})-(n-1)$. 
\item[(d)]  If $\kappa_1\le0\le\kappa_2$, then $C_1=\left(1+\frac{2}{n-1}\right)r_{\max}H_{\kappa_2}(r_{\max})-r_{\max}H_{\kappa_1}(r_{\max})$.
\end{itemize}
\smallskip
Of course when $C_1\le0$, we only get a trivial bound. However, depending on $\kappa_1$ and $\kappa_2$,   in all cases, there exists $r_0\in(0,\infty]$ such that  for   $r_{\max}<r_0$,  $C_1$ is positive. 
\par
\medskip
We proceed with the proof of the second inequality of part $i)$. Let $u_1,\cdots,u_k\in H^2(\Omega)$ be a family of  eigenfunctions associated to $\eta_1,\cdots,\eta_k$. We can choose $u_1,\cdots,u_k$ such that ${\partial_\nu u_1}, \cdots, {\partial_\nu u_k}$ are orthonormal in $L^2(\partial\Omega)$. Then, due to \eqref{dirvar} and \eqref{bihar1var}, we have
\ben{ah-0}\lambda_k\le\eta_k\sup_{u\in E_k}\frac{\int_{\partial\Omega}({\partial_\nu u})^2\,ds_g}{\int_{\Omega}|\nabla u|^2\,dv_g},\een
where $E_k:= \Span(u_1,\cdots,u_k)$.
Applying Theorem \ref{nrellich} with $\lambda=0$ and $F=\nabla\rho_p$ we get
\bq
\int_{\partial\Omega}\la\nabla \rho_p,\nu\ra({\partial_\nu u})^2\,ds_g&=&2\int_{\Omega}\Delta u\langle \nabla \rho_p,\nabla u\rangle dv_g-\int_{\Omega}\Delta \rho_p\lvert \nabla u\lvert^2dv_g+2\int_{\Omega}\nabla^2\rho_p(\nabla u,\nabla u)dv_g\\
&\le& 2\max_{x\in\Omega}|\nabla \rho_p|\left(\int_{\Omega}(\Delta u)^2dv_g\int_{\Omega}|\nabla u|^2 dv_g\right)^{1/2}\\
&&+\left(-1-\min_{x\in\Omega}d_p(x)H_{\kappa_2}(d_p(x))+2\max_{x\in\Omega}\frac{d_p(x)H_{\kappa_1}(d_p(x))}{n-1}\right)\int_{\Omega}|\nabla u|^2 dv_g\\
&\le&2r_{\max}\eta_k^{\frac{1}{2}}\left(\int_{\partial\Omega}({\partial_\nu u})^2\,ds_g\int_{\Omega}|\nabla u|^2 dv_g\right)^{1/2}
-C_2\int_{\Omega}|\nabla u|^2 dv_g,
\eq
where \begin{equation}\label{comega}C_2:=1+\min_{x\in\Omega}d_p(x)H_{\kappa_2}(d_p(x))-2   \max_{x\in\Omega}\frac{d_p(x)H_{\kappa_1}(d_p(x))}{n-1}.\end{equation} Let $A^2:=\frac{\int_{\partial\Omega}({\partial_\nu u})^2\,ds_g}{\int_{\Omega}|\nabla u|^2 dv_g}$. From the above inequality, $A$ satisfies
\[h_{\min}A^2\le 2r_{\max}\eta_k^{\frac{1}{2}}A-C_2.\]
This implies $$r_{\max}^2\eta_k-h_{\min}C_2\ge0,$$
Remark that since this is true for every $k$, we get  in particular
\bqn{eta1}\eta_1\ge\frac{h_{\min}C_2}{r^2_{\max}}.\eqn 
We now obtain the following upper bound on $A^2$
\[A^2\le\frac{\left(r_{\max}\eta_k^{\frac{1}{2}}+\sqrt{r_{\max}^2\eta_k-C_2 h_{\min}}\right)^2}{h_{\min}^2}\le
\frac{4r_{\max}^2\eta_k-2C_2 h_{\min}}{h_{\min}^2}.\]
Replacing in \eqref{ah-0} we conclude 
\[\lambda_k\le 
\frac{4r_{\max}^2\eta_k^2-2C_2 h_{\min}\eta_k      }{h_{\min}^2}.\]
\begin{rem}
The function $rH_{\kappa}(r)$ is constant if $\kappa=0$, increasing on $[0,\infty)$ if $\kappa<0$, and decreasing on $[0,\infty)$ if $\kappa>0$. We calculate $C_2$ considering different cases:
\begin{itemize}\label{escomega}
\item[(a)] If $\kappa_1=\kappa_2=0$, then $C_2=n-2$.
\item[(b)] If $\kappa_1\le\kappa_2\le0$, then $C_2=n-2 \frac{r_{\max}H_{\kappa_1}(r_{\max})}{n-1}$. 
\item[(c)] $0\le\kappa_1\le\kappa_2$.  Then $C_2=r_{\max}H_{\kappa_2}(r_{\max})-1$.
\item[(d)]  $\kappa_1\le0\le\kappa_2$.  Then $C_2=1+r_{\max}H_{\kappa_2}(r_{\max})-2 \frac{r_{\max}H_{\kappa_1}(r_{\max})}{n-1}$.
\end{itemize}
Depending on $\kappa_1$ and $\kappa_2$,   in all cases , there exists $r_0\in(0,\infty]$ so that  when $r_{\max}<r_0$, then $C_2$ is positive. 

\end{rem}
\smallskip

\item[\it ii)]
Let $\phi>0$ be a continuous function on $\partial\Omega$. For every $l\in \N$ set
\bq
\xi_{l+1}(\phi)&:=&\inf_{\substack{V\subset \tilde H_{N,\phi}^2(\Omega)\\\dim V=l}} \sup_{u\in V}\frac{\int_\Omega |\Delta u|^2\,dv_g}{\int_{\partial\Omega}u^2\phi\,ds_g},\qquad \xi_1(\phi)=0,\\
\eq
where $\tilde H_{N,\phi}^2(\Omega):=\{u\in H^2(\Omega)\,:\, {\partial_\nu u}=0\,\,\text{on $\partial \Omega$ and $\int_{\partial \Omega}\phi uds_g=0$}\}$. The following relation between $\xi_{l}$ and $\xi_l(\phi)$ holds:
\ben{bb}\xi_l\le \|\phi\|_\infty \xi_l(\phi).\een
Indeed, let $V=\Span(v_1,\cdots,v_l)$ be a subspace of $\tilde H_{N,\phi}^2(\Omega)$ of dimension $l$.
The functional space  $W=\Span(w_1,\cdots,w_l)$, where $w_j=v_j-\frac{1}{\vol(\partial\Omega)}\int v_j\,ds_g$, is an $l$-dimensional subspace of $\tilde H_{N}^2(\Omega):=\{u\in H^2(\Omega)\,:\, {\partial_\nu u}=0\,\,\text{on $\partial \Omega$ and $\int_{\partial \Omega}uds_g=0$}\}$ since $1\notin V$. It is easy to check that for every $v\in \tilde H_{N,\phi}^2(\Omega)$ and $w=v-\frac{1}{\vol(\partial\Omega)}\int v\,ds_g$ we have
\[\frac{\int_\Omega |\Delta w|^2\,dv_g}{\|\phi\|_\infty\int_{\partial\Omega}w^2\,ds_g}\le\frac{\int_\Omega |\Delta v|^2\,dv_g}{\int_{\partial\Omega}v^2\phi\,ds_g},\]
and inequality \eqref{bb} follows.
Later on we take $\phi:=\la\nabla \rho_p,\nu\ra$. Thus, it is enough to  show that  
\[\xi_{m+1}(\phi)\le \frac{\mu_k^2}{(C_3-n^{-1}\mu_kr_{\rm in}^2)\vee0},
\]
for some constants $C_3$.
Let $E_k$ be the eigenspace associated with $\mu_k$, $k\ge2$, and $u_1,\cdots,u_m$ be an orthonormal basis for $E_k$. Let $F$ be a vector field on $\Omega$ satisfying  properties A--C on page \pageref{i-iii}.
Consider
$$v_j:= u_j-\frac{1}{\int_\Omega\diver F\,dv_g}\int_{\partial\Omega}u_j\la F,\nu \ra ds_g,\qquad j=1,\cdots,m.$$
The functional space $V=\Span(v_1,\ldots,v_m)$ forms an $m$-dimensional subspace of $\tilde H_{N,\phi}^2(\Omega)$, where $\phi:=\la F,\nu \ra$. 
\bq
\xi_{m+1}(\phi)\le \sup_{v\in V}\frac{\int_\Omega |\Delta v|^2\,dv_g}{\int_{\partial\Omega}v^2 \la F,\nu \ra\,ds_g}= \sup_{u\in E_k}\frac{\mu_k^2\int_\Omega  u^2\,dv_g}{\int_{\partial\Omega}u^2\la F,\nu \ra\,ds_g-({\int_\Omega\diver F\,dv_g})^{-1}\left(\int_{\partial\Omega}u\la F,\nu \ra ds_g\right)^2}.\eq
By the Green formula and Theorem \ref{nrellich}, we get 
\bq
\int_{\partial\Omega}u^2\la F,\nu \ra\,ds_g&=&2\int_\Omega u\la\nabla u, F\ra dv_g+\int_\Omega u^2\diver F dv_g \\
&=&2\mu_k^{-1}\int_\Omega \Delta u\la\nabla u, F\ra dv_g+\int_\Omega u^2\diver F dv_g\\
&=& \mu_k^{-1}\left(\int_{\partial\Omega}\vert\nabla u\vert^2 \langle F,\nu\rangle ds_g-\int_{\Omega}\diver F\lvert \nabla u\lvert^2dv_g+2\int_{\Omega}DF(\nabla u,\nabla u)dv_g\right)\\
&&+\int_\Omega u^2\diver F dv_g\\
&\ge&\mu_k^{-1} \int_{\partial\Omega}\vert\nabla u\vert^2 \langle F,\nu\rangle ds_g+(c_1-c_2+2\alpha)\int_\Omega u^2\,dv_g\\&\ge&(c_1-c_2+2\alpha)\int_\Omega u^2\,dv_g.
\eq
We also have 
\bq
\left(\int_{\partial\Omega}u\la F,\nu \ra ds_g\right)^2&=&\left(\int_\Omega\la F,\nabla u\ra\,dv_g\right)^2\le\int_\Omega|F|^2\,dv_g\int_\Omega|\nabla u|^2\,dv_g\\
&=&\mu_k\int_\Omega|F|^2\,dv_g\int_\Omega u^2\,dv_g.
\eq
Therefore,
\bq
\xi_{m+1}(\phi)\le
\frac{\mu_k^2}{((c_1-c_2+2\alpha)-c_1^{-1}\vol(\Omega)^{-1}\mu_k\int_\Omega|F|^2\,dv_g)\vee0}.
\eq
Thanks to the Laplace and Hessian comparison theorem, the vector field $F=\nabla\rho_p$ satisfies properties $A-C$ (see page \pageref{i-iii}) on $\Omega$ with $\alpha=1$, and $$c_1= n,\quad\quad c_2=1+\max_{r\in[0,r_{\max})} rH_{\kappa}(r)=1+r_{\max}H_{\kappa}(r_{\max}).$$   Taking 
\bqn{ctilde}C_3:=n+1-r_{\max}H_{\kappa}(r_{\max}),\eqn
 we get
\bq
\xi_{m+1}(\phi)\le
\frac{\mu_k^2}{(C_3- n^{-1}\vol(\Omega)^{-1}\mu_k\int_\Omega d_p^2\, dv_g)\vee0}
\eq
 which completes the proof.  
\end{itemize}
\ef
Finally, we provide  examples for Theorem \ref{kutsig3} $(ii)$ in which vector fields satisfying conditions A-C arise naturally. The first example is just a special case  of Theorem \ref{kutsig3} $(ii)$. 

\begin{example}\label{rncase}
Let $\Omega$ be a star-shaped domain $\Omega$ in $\R^n$ with respect to the origin. Thus $F(x)=x$ satisfies properties $A-C$ above on $\Omega$ for $\alpha=1$ and $c_1=c_2=n$.  Then by Theorem \ref{kutsig3} part $ii$ we have
\bq
\xi_{m+1}\le
\frac{\max_{x\in \partial\Omega}\la x,\nu\ra\mu_k^2}{(2-n^{-1}\vol(\Omega)^{-1}\mu_kI_2(\Omega))\vee0},
\eq
where $m$ is the multiplicity of $\mu_k$ and $I_2(\Omega)=\int_\Omega|x|^2\,dv_g$ is the second moment of inertia. \\ If in addition the origin is also the centroid of $\Omega$, i.e. $\int_\Omega xdv_g=0$, then we have
\bq
\xi_{m_0+1}\le
{\max_{x\in \partial\Omega}\la x,\nu\ra\mu_2^2}, \eq 
where $m_0$ denotes the multiplicity of $\mu_2$. 
Combining this inequality with Theorem\,\ref{kutsig1} $(b)$ we get 
$$\sigma_{m_0+1}\le {\max_{x\in \partial\Omega}\la x,\nu\ra}{\mu_2}.$$
These two last inequalities has been previously obtained in \cite{KS68} for the special case $n=2$.
\end{example}

\begin{example}\label{minsurf} Let $\bf B^{n+1}$ be the unit ball in $\R^{n+1}$ centered at the origin, and $\Omega$ be a  free boundary minimal hypersurface in $\bf B^{n+1}$.  Consider $F(x)=x$, or equivalently $\rho_0(x)=\rho(x)=\frac{|x|^2}{2}$.  It is well-known that the coordinate functions of $\R^{n+1}$ are harmonic on $\Omega$. Hence  $$\diver F=\Delta\rho=n.$$
Also, by the definition of a free boundary minimal hypersurface, we have $\la\nabla\rho,\nu\ra=1$ on $\partial \Omega$. Thus,  conditions A and C on page \pageref{i-iii} are satisfied.  
To verify condition B, one can show that the eigenvalues of $\nabla^2\rho$ at point $x\in\Omega$ are given by $1-\kappa_i\la x, N(x)\ra$, $i=1,\cdots, n$, where $N(x)$ is the unit normal  to the $\Omega$ such that $N|_{\partial\Omega}=\nu$, and $\kappa_i$ are principal curvatures. Indeed, let $X,Y\in T_x\Omega$. Then we have
\begin{eqnarray*}
\nabla^2\rho(x)(X,Y)&=&X\cdot(Y\cdot\rho(x))-\nabla_XY\cdot\rho(x)\\
&=&X\la x,Y\ra-\la x,\nabla_XY\ra\\
&=&\la X,Y\ra+\la x,D_XY\ra-\la x,\nabla_XY\ra\\
&=&\la X,Y\ra-\la x,\la S(X),Y\ra N(x)\ra\\
&=&\la X-S(X),Y\ra\la x,N(x)\ra,
\end{eqnarray*} 
where $\la\cdot,\cdot\ra$ is the Euclidian inner product, $\nabla$ is the induced connection on $\Omega$, $D$ is the Euclidean connection (or simply the differentiation) on $\R^{n+1}$, and $S(x)$ is the shape operator 
$$S:T_x\Omega\to T_x\Omega,\quad X\mapsto \nabla_XN.$$
Then the eigenvalues of $\nabla^2\rho(x)$ are of the form  $1-\kappa_i(x)\la x,N(x)\ra$, $i=1,\ldots,n$. Define
$$\alpha:=\min_{\substack{i=1,\ldots,n\\x\in \Omega}}(1-\kappa_i\la x, N(x)\ra).$$
When $\alpha>0$, then $\Omega$ with vector field $F$ as above satisfies all three conditions A-C on page \pageref{i-iii}. In particular, following the proof of Theorem\,\ref{kutsig3} $ii$, we get
\bq
\xi_{m+1}\le
\frac{\mu_k^2}{(2\alpha-n^{-1}\vol(\Omega)^{-1}\mu_k\int_\Omega|x|^2\,dv_g)\vee0}.
\eq

In dimension two, $\alpha>0$ is equivalent to $|\kappa_i|\la x, N(x)\ra<1$. By results in \cite{AN16}, if $|\kappa_i|\la x, N(x)\ra<1$ then $\la x, N(x)\ra\equiv0$ on $\Omega$, and $\Omega$ is the equilateral disk. Hence, there is no nontrivial 2-dimensional minimal surface satisfying condition A-C.   It is an intriguing question whether there are non-trivial minimal hypersurfaces with $\alpha>0$ in higher dimensions. 
\end{example}


\end{document}